\documentclass[11pt]{amsart}

\usepackage{amsmath, amsfonts, amssymb,color, appendix, epsfig}

\def\k{\kappa}

\newcommand{\mB}{\mathcal{B}}

\newcommand{\fh}{\mathfrak{h}}

\newcommand{\bfC}{\mathbf{C}}

\newcommand{\bfQ}{\mathbf{Q}}
\newcommand{\bfR}{\mathbf{R}}

\newcommand{\bfZ}{\mathbf{Z}}

\newcommand{\AQ}{\mathbf{A}}

\newcommand{\ov}{\overline}

\newcommand{\be}{\begin{equation}}
\newcommand{\ee}{\end{equation}}
\newcommand{\bes}{\begin{equation*}}
\newcommand{\ees}{\end{equation*}}

\newcommand{\bs}{\begin{split}}
\newcommand{\es}{\end{split}}
\newcommand{\bss}{\begin{split*}}
\newcommand{\ess}{\end{split*}}

\newcommand{\bmat}{\left[ \begin{matrix}}
\newcommand{\emat}{\end{matrix} \right]}
\newcommand{\bsmat}{\left[ \begin{smallmatrix}}
\newcommand{\esmat}{\end{smallmatrix} \right]}

\newcommand{\bml}{\begin{multline}}
\newcommand{\eml}{\end{multline}}
\newcommand{\bmls}{\begin{multline*}}
\newcommand{\emls}{\end{multline*}}

\DeclareMathOperator{\GL}{GL}
\DeclareMathOperator{\GSp}{GSp}

\DeclareMathOperator{\ord}{ord}

\DeclareMathOperator{\SL}{SL}

\DeclareMathOperator{\Sym}{Sym}

\DeclareMathOperator{\vol}{vol}

\def\c{\chi}

\newcommand{\hs}{\hspace{2pt}}

\newcommand{\tr}{\textup{tr}\hspace{2pt}}

\usepackage[all]{xy}
\usepackage{amsthm}
\usepackage{amssymb, amsfonts}
\SelectTips{cm}{10}\UseTips
\theoremstyle{plain}
\newtheorem{thm}{Theorem}

\newtheorem{cor}[thm]{Corollary}
\newtheorem{lemma}[thm]{Lemma}

\theoremstyle{definition}

\newtheorem{example}[thm]{Example}
\newtheorem{rem}[thm]{Remark}

\numberwithin{thm}{section}
\numberwithin{equation}{section}

\newcommand{\mat}[4]{\left[ \begin{matrix} #1 & #2 \\ #3 & #4 \end{matrix} \right]}

\newcommand{\pp}{\mathfrak p}

\newcommand{\ZZ}{\mathbf Z}
\newcommand{\QQ}{\mathbf Q}

\newcommand{\re}{\operatorname{Re}}

\newcommand{\Nm}{{\rm N}}

\makeatletter
\let\@wraptoccontribs\wraptoccontribs
\makeatother

\begin{document}

\title{On the norms of $p$-stabilized elliptic newforms}
\author[Jim Brown]{Jim Brown$^1$}
\address{$^1$Department of Mathematical Sciences\\
Clemson University\\ Clemson, SC 29634}
\email{jimlb@clemson.edu}
\author[Krzysztof Klosin]{Krzysztof Klosin$^2$}
\address{$^2$Department of Mathematics\\
Queens College\\
City University of New York\\
Flushing, NY 11367}
\email{kklosin@qc.cuny.edu}
\contrib[with an appendix by]{Keith Conrad$^3$}
\address{$^3$Department of Mathematics\\
University of Connecticut\\
Storrs, CT 06269}
\email{kconrad@math.uconn.edu}

\subjclass[2010]{Primary 11F67; Secondary 11F11}
\keywords{$p$-stabilized forms, modular periods, symmetric square $L$-function, Euler products}

\thanks{The first author was partially supported by the National Security Agency under Grant Number H98230-11-1-0137.  The United States Government is authorized to reproduce and distribute reprints not-withstanding any copyright notation herein. The second author was partially supported by a PSC-CUNY Award, jointly funded by The Professional Staff Congress and The City University of New York.}

\maketitle

\begin{abstract}
Let $f \in S_{\kappa}(\Gamma_0(N))$ be a Hecke eigenform at $p$ with eigenvalue $\lambda_f(p)$ for a prime $p \nmid N$.  Let $\alpha_p$ and $\beta_p$ be complex numbers satisfying $\alpha_p + \beta_p = \lambda_{f}(p)$ and $\alpha_p \beta_p = p^{\k-1}$. We calculate the norm of $f_{p}^{\alpha_p}(z) = f(z) - \beta_{p} f(pz)$ as well as the norm of $U_p f$, both classically and adelically.  We use these results along with some convergence properties of the Euler product defining the symmetric square $L$-function of $f$ to give a `local' factorization of the Petersson norm of $f$.
\end{abstract}

\section{Introduction}
Let $\k \geq 2$
and $N \geq 1$ be integers and $p$ an odd prime with $p \nmid N$.  Let $f \in S_{\k}(\Gamma_0(N))$ be a newform. It is well-known that the Petersson norm $\langle f,f \rangle$ serves as a natural period for many $L$-functions of $f$ \cite{HidaInvent81,ShimuraMathAnn77}.

In this paper we focus on related periods $\langle f_p^{\alpha_p}, f_p^{\alpha_p} \rangle$ (defined below) for $\alpha_p$ a Satake parameter of $f$. When $f$ is ordinary at $p$, the forms $f_p^{\alpha_p}$ arise naturally in the context of Iwasawa theory as the objects which can be interpolated into a Hida family. It is in fact in the context of `$p$-adic interpolation' of some automorphic lifting procedures (between two
algebraic groups, one of them being $\GL_2$)  that these calculations arise (see \cite{BrownKlosinGSp(4)} for example);
 however, our results apply in a more general setup as specified below.

Let $f \in S_{\k}(\Gamma_0(N))$ be an eigenform for the $T_p$-operator with eigenvalue $\lambda_{f}(p)$.  Let  $\alpha_p$ and $\beta_p$ be the pair of complex numbers satisfying $\alpha_p + \beta_p = \lambda_f(p)$ and $\alpha_p \beta_p  = p^{\k-1}$.  We set $f_{p}^{\alpha_{p}}(z) = f(z) - \beta_{p} f(pz)$.  In the case that $f$ is ordinary at $p$, we can choose $\alpha_p$ and $\beta_p$ so that $\alpha_{p}$ is a $p$-unit and $\beta_{p}$ is divisible by $p$.
 In this special case  $f_p^{\alpha_p}$ is the $p$-stabilized ordinary newform of tame level $N$ attached to $f$.

Since $f_p^{\alpha_p} = p^{1-\kappa}\beta_p (U_p-\beta_p)f$, calculating $\langle f_p^{\alpha_p}, f_p^{\alpha_p} \rangle$ is in fact equivalent to calculating $\langle U_p f, U_p f \rangle$. While computation of any of these inner products does not present any difficulties  (see Section \ref{Relations between the norms}), it is an accident resulting from the relative simplicity of the Hecke algebra on $\GL_2$, where the $T_p$ and the $U_p$ operators differ by a single term. It turns out that in the higher-rank case it is the calculation of the latter inner product that provides the fastest route to computing the Petersson norm of various $p$-stabilizations. With these future applications in mind we present an alternative approach to calculating $\langle U_p f, U_p f \rangle$, this time working adelically (see Sections \ref{relations} and \ref{sec:adeliccalc}), as this is the method that generalizes to higher genus most readily (see \cite{BrownKlosinGSp(4)}, where this is done for the group $\GSp_4$).

It is well-known
that the Petersson norm $\langle f,f \rangle$ is closely related to the value $L(\kappa, \Sym^2 f)$ at $\kappa$ of the symmetric square $L$-function of $f$.
The absolutely convergent Euler product defining this $L$-function for $\re (s) >\kappa$ converges (conditionally) to the value $L(s, \Sym^2 f)$ when $\re (s)=\kappa$ (this and in fact a more general result is proved in the appendix by Keith Conrad).
On the other hand our computation of $\langle f_p^{\alpha_p}, f_p^{\alpha_p} \rangle$ shows that this inner product differs from $\langle f,f \rangle$ by essentially the $p$-Euler factor of $L(\kappa, \Sym^2 f)$. Combining these facts we exhibit a (conditionally convergent) factorization of $\langle f,f \rangle$ into local components defined via the inner products $\langle f_p^{\alpha_p}, f_p^{\alpha_p} \rangle$ (for details, see Section \ref{Applications to $L$-values}).

The authors would like to thank Henryk Iwaniec, and Keith Conrad would like to thank Gergely Harcos, for helpful email correspondence.

\section{Classical calculation of $\langle f_p, f_p \rangle$ and $\langle U_p f, U_p f \rangle$} \label{Relations between the norms}
Let $N$ be a positive integer.
Let $\Gamma_0(N)\subset \SL_2(\bfZ)$ denote the subgroup consisting of matrices whose lower-left entry is divisible by $N$.
For a holomorphic function $f$ on the complex upper half-plane
$\fh$ and for $\gamma=\bmat a&b \\ c&d \emat \in \GL_2^+(\bfR)$,
where $+$ denotes positive determinant,
 and $\kappa \in \bfZ_+$ we define the slash operator as
\begin{equation*}
(f|_{\k} \gamma)(z) = \frac{\det(\gamma)^{\kappa/2}}{(cz+d)^\kappa} f\left(\frac{az+b}{cz+d}\right).
\end{equation*}
If $\kappa$ is clear from the context we will simply write $f|\gamma$ instead of $f|_{\kappa}\gamma$.
We will write $S_{\kappa}(\Gamma_0(N))$ for the $\bfC$-space of cusp forms of weight $\kappa$ and level $\Gamma_0(N)$ (i.e., functions $f$ as above which satisfy $f|_{\kappa} \gamma = f$ for all $\gamma \in \Gamma_0(N)$ and vanish at the cusps - for details see \cite{Miyake89}).

The space $S_{\kappa}(\Gamma_0(N))$ is endowed with a natural inner product (the \emph{Petersson inner product}) defined by
\begin{equation*}
\langle f, g \rangle_{N} = \int_{\Gamma_0(N) \backslash \fh} f(z) \overline{g(z)} y^{\k-2} dx dy
\end{equation*}
for $z = x+iy$ with $x,y \in \bfR$ and $y>0$.
If $\Gamma \subset \Gamma_0(N)$ is a finite index subgroup we also set $$\langle f, g \rangle_{\Gamma} = \int_{\Gamma \backslash \fh} f(z) \overline{g(z)} y^{\k-2} dx dy.$$

From now on let $p$ be a prime which does not divide $N$. Set $\eta = \mat{p}{0}{0}{1}$. We have the decomposition
\begin{equation} \label{decomp3}
\Gamma_0(N) \mat{1}{0}{0}{p} \Gamma_0(N) = \bigsqcup_{j=0}^{p-1} \Gamma_0(N) \mat{1}{j}{0}{p} \sqcup \Gamma_0(N) \eta.
\end{equation}

Recall
the $p$th Hecke operator acting on $S_{\k}(\Gamma_0(N))$ is given by
\begin{equation*}
T_p f = p^{\k/2-1}\left( \sum_{j=0}^{p-1} f|_{\k}  \mat{1}{j}{0}{p} + f|_{\k} \eta\right)
\end{equation*}
and the $p$th Hecke operator acting on $S_{\k}(\Gamma_0(Np))$ is given by
\begin{equation*}
U_pf =  p^{\k/2-1}\left( \sum_{j=0}^{p-1} f|_{\k}  \mat{1}{j}{0}{p} \right).
\end{equation*}
As we will be viewing $f \in S_{\k}(\Gamma_0(N))$ as an element of $S_{\k}(\Gamma_0(Np))$, we use $T_p$ and $U_p$ to distinguish the two Hecke operators at $p$ defined above.

Let $f \in S_{\k}(\Gamma_0(N))$ be an eigenfunction for $T_p$ with eigenvalue $\lambda_f(p)$.  There exist (up to permutation) unique complex numbers $\alpha_p$ and $\beta_p$ satisfying $\lambda_f(p) = \alpha_p + \beta _p$ and $\alpha_p\beta_p=p^{\kappa-1}$. We consider the following two forms:
\begin{align*}
f_{p}^{\alpha_p}(z) &= f(z) - \beta_{p} p^{-\k/2} (f|_{\kappa}\eta)(z),\\
f_{p}^{\beta_p}(z) &= f(z) - \alpha_{p} p^{-\k/2} (f|_{\kappa}\eta)(z).
\end{align*}
One immediately obtains that $f_p^{\alpha_p} \in S_{\k}(\Gamma_0(Np))$ and that $f_p^{\alpha_p}$ is an eigenfunction for the operator $U_p$ with eigenvalue $\alpha_p$. Furthermore,
if $f$ is also an eigenform for $T_{\ell}$ for a prime $\ell \neq p$, then so is $f_p^{\alpha_p}$ and it has the same $T_{\ell}$-eigenvalue as $f$.
The analogous statements for $f_{p}^{\beta_p}$ hold as well.
Note that if $f$ is ordinary at $p$, then one can choose $\alpha_p$ and $\beta_p$ so that $\ord_p(\alpha_p) = 0$ and then $f_p^{\alpha_p}$ is the $p$-stabilized newform associated to $f$, see \cite{WilesInventMath88} for example.

\begin{thm} \label{classform} Let $f \in S_{\k}(\Gamma_0(N))$ be defined as above, where $p \nmid N$.  We have
\begin{equation*}
 \frac{\langle U_p f, U_p f \rangle_{Np}}{\langle f, f \rangle_{Np}} = p^{\k-2} + \frac{(p-1) \lambda_{f}(p)^2}{p+1}
 \end{equation*}
 and
\begin{equation*}
\frac{\langle f_p^{\alpha_p}, f_p^{\alpha_p} \rangle_{Np}}{\langle f, f \rangle_{Np}} =  \frac{p}{p+1}\left(1 - \frac{\alpha_p^2}{p^\kappa}\right)\left(1 - \frac{\beta_p^2}{p^\kappa}\right).
 \end{equation*}
 \end{thm}

\begin{proof}
The definition of $f_{p}^{\alpha_p}$
and the fact
 that $U_p f = T_{p} f - p^{\k/2-1}f|\eta$
immediately give
\begin{align*}
\langle f_p^{\alpha_p}, f_p^{\alpha_p} \rangle_{Np} &=
  (1 + |\beta_{p}|^2 p^{-\k}) \langle f, f \rangle_{Np} - p^{-\k/2}(\beta_{p} \langle f| \eta, f \rangle_{Np} + \ov{\beta_{p} \langle f|\eta, f\rangle}_{Np})
 \end{align*}
and
\begin{align*}
\langle U_p f, U_p f \rangle_{Np}
    &= (p^{\k-2} + \lambda_{f}(p)^2) \langle f, f \rangle_{Np} -   p^{\k/2-1}\lambda_{f}(p) (\langle f|\eta, f \rangle_{Np} + \ov{\langle f|\eta, f\rangle}_{Np}).
\end{align*}

Let us now compute $\langle f|\eta, f \rangle_{Np}$.
Observe that by the definition of $T_p$ we have
\begin{equation*}
\langle T_p f, g \rangle_{Np} =  p^{\k/2-1}\left( \sum_{j=0}^{p-1} \left \langle f\mid \mat{1}{j}{0}{p}, g\right \rangle_{Np} + \langle f|\eta, g \rangle_{Np}\right).
\end{equation*}
Using the decomposition (\ref{decomp3})
we can find $a_{j}, b_{j} \in \Gamma_0(N)$ so that $a_{j} \mat{1}{j}{0}{p} b_{j} = \mat{1}{0}{0}{p}$, and $a,b \in \Gamma_0(N)$ so that
$a \mat{1}{0}{0}{p} b = \mat{p}{0}{0}{1}.$
Using this and the fact that $f, g \in S_{\k}(\Gamma_0(N))$, we have
\begin{align*}
p^{1-\k/2} \langle  T_p f,g \rangle_{Np}
    &= \sum_{j=0}^{p-1} \left \langle f\mid a_{j} \mat{1}{j}{0}{p}, g|b_{j}^{-1} \right \rangle_{Np} + \langle f|\eta, g \rangle_{Np} \\
    &=\sum_{j=0}^{p-1} \left \langle f\mid a_{j}\mat{1}{j}{0}{p} b_{j}, g \right \rangle_{Np} + \langle f|\eta, g \rangle_{Np} \\
    &= p \left \langle f\mid \mat{1}{0}{0}{p}, g \right \rangle_{Np} + \langle f|\eta, g \rangle_{Np} \\
    &= p \left \langle f\mid a \mat{1}{0}{0}{p}, g|b^{-1} \right \rangle_{Np} + \langle f|\eta, g \rangle_{Np}\\
    &= (p+1) \langle f|\eta, g \rangle_{Np}.
 \end{align*}
Thus, setting $g = f$ we obtain
 \begin{equation*}
 \langle f|\eta, f \rangle_{Np} =  p^{1-\k/2}\frac{\lambda_{f}(p)}{p+1} \langle f, f \rangle_{Np}.
 \end{equation*}
 We can now easily conclude that
 \begin{equation*}
 \frac{\langle f_p^{\alpha_p}, f_p^{\alpha_p} \rangle_{Np}}{\langle f, f \rangle_{Np}} = 1 + |\beta_{p}|^2 p^{-\k} - p^{1-\k}(\beta_{p}+\ov{\beta_p})\frac{\lambda_{f}(p)}{p+1}
 \end{equation*}
and
 \begin{equation*}
 \frac{\langle U_p f, U_p f \rangle_{Np}}{\langle f, f \rangle_{Np}} = p^{\k-2} + \frac{(p-1) \lambda_{f}(p)^2}{p+1}.
 \end{equation*}

Using the fact
that $T_p$ is self-adjoint with respect to the Petersson inner
product we have $\alpha_p + \beta_p = \lambda_f(p) \in \bfR$.
We note by Lemma \ref{Ramanujan} below that $\ov{\alpha_{p}} = \beta_{p}$.  Thus $|\alpha_p|^2=|\beta_p|^2 = |\alpha_p||\beta_p|=p^{\kappa-1}$ and $\beta_p + \ov{\beta_p}=\lambda_f(p)$. This allows us to simplify the formula for $ \frac{\langle f_p^{\alpha_p}, f_p^{\alpha_p} \rangle_{Np}}{\langle f, f \rangle_{Np}}$ to
 \begin{equation*}
 \frac{\langle f_p^{\alpha_p}, f_p^{\alpha_p} \rangle_{Np}}{\langle f, f \rangle_{Np}} = 1 + \frac{1}{p} - p^{1-\k}\frac{\lambda_{f}(p)^2}{p+1}.
 \end{equation*}
Again using that $\lambda_f(p) = \alpha_p+\beta_p$ we obtain
\begin{align*}
\frac{\langle f_p^{\alpha_p}, f_p^{\alpha_p} \rangle_{Np}}{\langle f, f \rangle_{Np}} &= \frac{1}{p+1}\left(p+ 1+\frac{p+1}{p} - p^{1-\kappa}(\alpha_p^2+\beta_p^2 + 2p^{\kappa-1})\right)\\
&= \frac{p}{p+1}\left(1-\frac{\alpha_p^2}{p^{\kappa}}\right)\left(1-\frac{\beta_p^2 }{p^{\kappa}}\right).
\end{align*}
\end{proof}

\begin{cor}  We have
\begin{equation*}
\lim_{\substack{p\to \infty\\ p\hs\textup{is prime}}} \frac{\langle f_p^{\alpha_p}, f_p^{\alpha_p} \rangle_{Np}}{p+1}=\langle f, f \rangle_N.
\end{equation*} \end{cor}

\begin{proof}Using Theorem \ref{classform} and the fact that  $\langle f,f\rangle_{Np} = (p+1)\langle f, f\rangle_N$, we have for every prime $p \nmid N$ that $$\frac{\langle f_p^{\alpha_p}, f_p^{\alpha_p} \rangle_{Np}}{p+1}= \frac{p}{p+1}\left(1-\frac{\alpha_p^2}{p^{\kappa}}\right)\left(1-\frac{\beta_p^2 }{p^{\kappa}}\right)\langle f, f \rangle_{N}.$$ Since $|\alpha_p|^2=|\beta_p|^2=p^{\kappa-1}$, we see that the first three factors on the right tend to 1 as $p$ tends to infinity.
 \end{proof}

\section{Relation between the classical and adelic inner products} \label{relations}

While the classical calculations for $\langle U_p f, U_p f \rangle$ are rather elementary, it is also useful to note that one can perform these calculations adelically.  The problem of calculating $\langle U_p f, U_p f \rangle$ is one that is local in nature, so it lends itself nicely to such an approach. Moreover, in a higher genus setting such as when working with Siegel modular forms, it is the adelic approach that generalizes most readily \cite{BrownKlosinGSp(4)}.  In this section we provide the necessary background relating the adelic and classical inner products that is needed to relate the adelic inner product calculated in Section \ref{sec:adeliccalc} to the calculation given in the previous section.

In this and the following sections $p$ will denote a prime number and $v$ will denote an arbitrary place of $\bfQ$ including the Archimedean one, which we will denote by $\infty$.
Let $G=\GL_2$ and fix $N \geq 1$.   By strong approximation (see for example
 \cite[Theorem 3.3.1, p. 293]{BumpCambridgeUniversityPress88}) we have
\begin{equation}\label{strapp}
G(\AQ) = G(\bfQ) G(\bfR) \prod_{p} K_{p},
 \end{equation}
where $K_p$ is a compact subgroup of $G(\bfQ_p)$ such that $\det K_p = \bfZ_p^{\times}$. One example would be to take  $K_{p} = K_{0}(N)_p$, where
\begin{equation*}
K_0(N)_p = \left\{ \bmat a& b \\ c& d \emat \in G(\bfZ_{p}): c \equiv 0 \bmod{N}\right\}.
\end{equation*}
Note that $K_0(N)_p=G(\bfZ_p)$ if $p \nmid N$. We will also set $$K_0(N):= \left\{ \bmat a& b \\ c& d \emat \in G(\hat{\bfZ}): c \equiv 0 \bmod{N}\right\}=\prod_p K_0(N)_p.$$
The decomposition (\ref{strapp}) implies
 that
\be \label{2}
G(\bfQ) \setminus G(\AQ) =  G^+(\bfR) \prod_{p} K_{p},
\ee
where $+$
indicates positive determinant.

Let $Z \subset G$ denote the center.  For every $p$ there is a unique Haar measure $dg_p$ on $G(\bfQ_p)$ normalized so that the volume of any maximal compact subgroup of $G(\bfQ_p)$ is one.  We use the standard Haar measure on $G(\bfR)$ as defined in \cite[$\S$ 2.1]{BumpCambridgeUniversityPress88}.   Define the
 adelic analogue of the Petersson inner product:
\begin{equation*}
\langle \phi_1, \phi_2 \rangle = \int_{Z(\AQ) G(\bfQ) \setminus G(\AQ)} \phi_1(g) \ov{\phi_2(g)} dg,
 \end{equation*}
where $\phi_1$ and  $\phi_2$ lie in $L^2(Z(\AQ) G(\bfQ) \setminus G(\AQ))$ and have the same central character and $dg$ is
the Haar measure on $Z(\AQ) G(\bfQ) \setminus G(\AQ)$ corresponding to our choice of local Haar measures.

Let $f \in S_{\kappa}(\Gamma_0(N))$ be an eigenform.  For $g=\gamma g_{\infty} k\in G(\AQ)$ with $\gamma \in G(\bfQ)$, $g_{\infty}=\bmat a&b\\c&d\emat \in G^+(\bfR)$ and $k \in K_0(N)$, set \be \label{adelicdef}
 \phi_f(g)=\frac{(\det g_\infty)^{\kappa/2}}{(ci+d)^\kappa}f(g_\infty{i}).\ee
Then $\phi_f$ is an automorphic form on $G(\AQ)$ and it is easy to see (using the bijection in \cite[Equation 5.13]{GelbartAFAG}) that one has
\begin{equation}\label{eqn:classicalrelation}
\langle \phi_f, \phi_{f} \rangle = \frac{1}{\left[\SL_2(\bfZ):\Gamma_0(N)\right]} \langle f, f \rangle_{N}.
\end{equation}

Let $\pi_{f} \cong \otimes \pi_{f,v}$ be the automorphic representation generated
by $\phi_{f}$. If $f$ is a newform, then we can write $\phi_f= \otimes_v \phi_{f,v}$ for $\phi_{f,v} \in \pi_{f,v}$ and $\phi_{f,v}$ are spherical vectors for all $v \nmid N$, $v \neq \infty$. For every $v$ we can choose a $G(\bfQ_v)$-invariant inner product $\langle\cdot, \cdot\rangle_v$ (and any two such are scalar multiples of each other)
so that $\langle \phi_{f,v}, \phi_{f,v}\rangle_v=1$ for all $v \nmid N$, $v \neq \infty$.
 It follows that there is constant $c$
so that \begin{equation}\label{eqn:factorization}\langle \phi_{f}, \phi_{f}\rangle = c \prod_{v} \langle \phi_{f,v}, \phi_{f,v} \rangle_{v}. \end{equation} 
 We are now in a position to relate the ratio $\frac{\langle U_p f, U_p f \rangle_{Np}}{\langle f, f \rangle_{Np}}$ to something that can be calculated locally. In fact since we are only interested in this ratio, the precise value of the constant $c$ in (\ref{eqn:factorization}) will be irrelevant.

Fix $p \nmid N$.  As noted above, we normalize our Haar measure so that $\vol(K_0(1)_p) = 1$. For a vector $v_p$ inside the space of $\pi_{f,p}$, we set
\begin{equation*}
T_p v_p = \int_{K_0(1)_p \bsmat p&0 \\ 0& 1 \esmat K_0(1)_p} \pi_{f,p}(g) v_{p} dg
\end{equation*}
and
\begin{equation*}
 V_p v_{p} = \int_{\bsmat 1 & 0 \\ 0 & p \esmat K_0(1)_p} \pi_{f,p}(g) v_{p} dg.
 \end{equation*}
 Note that we have the decompositions
\be
\label{dec} K_0(1)_p \bmat p&0 \\ 0& 1 \emat K_0(1)_p = \bigsqcup_{b=0}^{p-1} \bmat p & b \\ 0 & 1 \emat K_0(1)_p \sqcup \bmat 1&0 \\ 0& p\emat K_0(1)_p
\ee
and
\be
\label{dec1} K_0(p)_p \bmat p&0 \\ 0& 1 \emat K_0(p)_p= \bigsqcup_{b=0}^{p-1} \bmat p & b \\ 0 & 1 \emat K_0(p)_p.
\ee
The adelic operator corresponding to the $U_p$-operator acting on classical modular forms as defined in Section \ref{Relations between the norms} is given by
$$U_p v_{p} := T_p v_{p} - V_p v_{p}.$$

\begin{lemma} \label{classtoadel} We have
 \begin{equation*}
 \frac{\langle U_p^{\rm cl} f, U_p^{\rm cl} f \rangle_{Np}}{\langle f, f \rangle_{Np}} = p^{\k-2}\langle U_p \phi_{f,p}, U_p \phi_{f,p} \rangle_{p}
 \end{equation*}
 for any local inner product pairing $\langle , \rangle_{p}$ so that $\langle \phi_{f,p}, \phi_{f,p} \rangle_{p} = 1$ and $U_p^{\rm cl}$ is the classical $U_p$-operator as defined in Section \ref{Relations between the norms}.
 \end{lemma}

\begin{proof} If we set $U_p \phi_f := (U_p \phi_{f,p}) \otimes \otimes_{v \neq p} \phi_{f,v}$ then it follows by the same argument as the one in the proof of \cite[Lemma 3.7]{GelbartAFAG} that
 \be \label{adelclas}
 \phi_{U_p^{\rm cl}f}=p^{\kappa/2-1}U_p \phi_f.
 \ee
The lemma is now immediate from (\ref{eqn:classicalrelation}) and   (\ref{eqn:factorization}).
\end{proof}

It only remains to calculate $\langle U_p \phi_{f,p}, U_p \phi_{f,p} \rangle$, which is done in the next section.

\section{Local calculation of $\langle U_p \phi_{f,p}, U_p \phi_{f,p} \rangle$}\label{sec:adeliccalc}

We will now give a calculation that, when combined with the results of the previous section, provides a local way to calculate $\langle U_p f, U_p f \rangle$ in terms of $\langle f, f \rangle$.  As in the previous section we fix $f \in S_{k}(\Gamma_0(N))$ a newform and a prime $p$ not dividing $N$.  We again let $\pi_{f} = \otimes_{v} \pi_{f,v}$ be the automorphic representation associated to $f$. Note that since $p \nmid N$, we can take the principal series representation  $\pi_{p}(\chi_1, \chi_2)$ to be the model for $\pi_{f,p}$ and for functions $\psi, \psi' \in \pi_p(\chi_1, \chi_2)$ define the local inner product by $$\langle \psi, \psi' \rangle_p := \int_{K_0(N)_p} \psi(g) \ov{\psi'(g)}dg,$$ where the Haar measure is normalized so that $\vol(K_0(N)_p)=1$. Then the vector $\phi_{f,p}\in \pi_p$ corresponds to the function (which we will also denote by $\phi_{f,p} \in \pi_p(\chi_1,\chi_2)$) which can be described explicitly as $$\phi_{f,p}\left( \bmat a &* \\ 0&b \emat k\right) = \chi_1(a) \chi_2(b)|ab^{-1}|_p^{1/2},$$
where $|\cdot|_p$ denotes the standard $p$-adic norm ($|p|_p=p^{-1}$) and $k \in K_0(N)_p$.

As this section is focused on the calculation of $\langle U_p \phi_{f,p}, U_p \phi_{f,p} \rangle$,
we will from now on write $\phi$ for $\phi_{f,p}$ and $K_0(1)$ (resp., $K_0(p)$) for $K_0(1)_p$ (resp., $K_0(p)_p$).  

\begin{rem} We note here that the calculation which follows can also be performed using the MacDonald formula for matrix coefficients (see \cite[$\S$ 4]{CasselmanCompMath80}). However, in the relatively simple case of $\GL_2$  the elementary approach which we present below does not add any computational difficulty and is perhaps more transparent. \end{rem}

Set $$\mB:= \left\{ \bmat p & b \\ 0 & 1 \emat : b\in \{0,1, \dots, p-1\}\right\}, \quad \mB':= \mB \cup \left\{ \bmat 1&0 \\ 0&p \emat \right\}.$$
If $g \in K_0(1)$ and $\beta \in \mB'$, there is a permutation $\sigma_{g}$ of $\mB'$ and elements $k(g,\beta) \in K_0(1)$ such that $g \beta = \sigma_{g}(\beta) k(g, \beta)$. Furthermore, note that if $g \in K_0(1) - K_0(p)$, then the corresponding permutation cannot fix  $\bmat 1&0 \\ 0& p \emat$. This implies that for such a $g$, there exists $\beta \in \mB$ such that $\sigma_{g}(\beta) = \bmat 1&0 \\ 0& p \emat$. Since in the computation of $U_p\phi$ only matrices in $\mB$ are used, we are interested in the restriction of $\sigma$ to $\mB$.
For such a $g$ there are $p-1$ matrices in the image of $\sigma_{g}$
which have $(p,1)$ on the diagonal and one that  has $(1,p)$ on the diagonal.
   Set $\mB_1(g) = \{\beta \in \mB: \sigma_{g}(\beta) \in \mB\}$ and $\mB_2(g) = \{\beta \in \mB: \sigma_{g}(\beta) \in \mB' - \mB\}$.  So for $g \in K_0(1) - K_0(p)$  we have (note that our $\phi$ is right-$K_0(1)$-invariant and $\vol(K_0(1))=1$)
\begin{equation*}
\begin{split} (U_p\phi)(g) &= \vol(K_0(1)) \sum_{\beta \in \mB} \phi(g\beta)\\
&= \sum_{\beta \in \mB} \phi(\sigma_{g}(\beta) k(g,\beta)) \\
&= \sum_{\beta \in \mB_1(g)} \phi(\sigma_{g}(\beta)) + \sum_{\beta \in \mB_2(g)} \phi(\sigma_{g}(\beta))\\
& =  (p-1)\chi_1(p) p^{-1/2} + \chi_2(p) p^{1/2}.
\end{split}
\end{equation*}
If $g \in K_0(p)$, then the permutation $\sigma$ 
fixes $\bmat 1&0 \\0 & p \emat$, hence we obtain
$$(U_p\phi)(g) = \vol(K_0(1)) \sum_{\beta \in \mB} \phi(g\beta) = \sum_{\beta \in \mB} \phi(\sigma(\beta))=p\chi_1(p)p^{-1/2} = \chi_1(p)p^{1/2}.$$
Now let us compute the integral:
$$ \left< U_p \phi, U_p \phi \right>_{K_0(1)} = \int_{K_0(p)} U_p\phi(g) \ov{U_p\phi(g)} dg + \int_{K_0(1) - K_0(p)} U_p\phi(g) \ov{U_p\phi(g)} dg.$$

%

We have \be \label{j1}
\begin{split} \int_{K_0(p)} U_p\phi(g) \ov{U_p\phi(g)} dg  &=  \int_{K_0(p)} p |\chi_1(p)|^2  dh \\
&=\vol(K_0(p)) p |\chi_1(p)|^2 \\
&= \frac{p |\chi_1(p)|^2}{p+1}
\end{split}
\ee and, since $\vol(K_0(1) - K_0(p))=p/(p+1)$,
\begin{multline} \label{k1} \int_{K_0(1) - K_0(p)} U_p\phi(g) \ov{U_p\phi(g)} dg\\
= \int_{K_0(1)- K_0(p)} \left[ \frac{(p-1)^2}{p} |\chi_1(p)|^2 + (p-1)\tr(\chi_1(p) \ov{\chi_2(p)})+ p|\chi_2(p)|^2\right] dg \\
= \frac{p}{p+1} \left[ \frac{(p-1)^2}{p} |\chi_1(p)|^2 + (p-1)\tr(\chi_1(p) \ov{\chi_2(p)})+ p|\chi_2(p)|^2\right].
\end{multline}
Putting (\ref{j1}) and (\ref{k1}) together we get
\be \label{for32}
\begin{split}
\left< U_p \phi, U_p \phi \right> &= \frac{p^2-p+1}{p+1} |\chi_1(p)|^2 + \frac{p^2}{p+1} |\chi_2(p)|^2+ \frac{p^2-p}{p+1} \tr (\chi_1(p)\ov{\chi_2(p)}). \end{split}
\ee

\begin{lemma} \label{Ramanujan} We have $\chi_j(p) = p^{s_j}$ for $j=1,2$, where $s_j$ is a purely imaginary number. In particular, $|\chi_{j}(p)|=1$ for $j=1,2$. Moreover, we have $\ov{\alpha_p} = \beta_{p}$.\end{lemma}

\begin{proof} The first part follows from \cite[p. 92]{GelbartAFAG} and is a direct consequence of the fact that cusp forms on $\GL_2$ satisfy the Ramanujan conjecture.  Observe that $\alpha_p = p^{(\k-1)/2} \chi_1(p)$ and $\beta_{p} = p^{(\k-1)/2} \chi_2(p)$.  Using that $\alpha_p \beta_{p} = p^{\k-1}$, we obtain $\chi_1(p) \chi_2(p) = 1$.  This, combined with the fact that $\chi_{j}(p) = p^{s_{j}}$ with $s_{j}$ purely imaginary implies $\chi_1(p) = \ov{\chi_2(p)}$.  Thus $\ov{\alpha_p} = \beta_p$.  \end{proof}

Using Lemma \ref{Ramanujan} we can simplify 
(\ref{for32}) to
\begin{equation*}
\langle U_p \phi, U_p \phi \rangle = \frac{2p^2-p+1}{p+1} + \frac{p^2-p}{p+1} \tr (\chi_1(p)\ov{\chi_2(p)}).
\end{equation*}
Moreover using that $\alpha_p = p^{(\kappa-1)/2}\chi_1(p)$, $\beta_p=p^{(\kappa-1)/2}\chi_2(p)$
and $\chi_1(p) = \ov{\chi_2(p)}$ we have
$$\tr (\chi_1(p)\ov{\chi_2(p)}) = 
 (\chi_1(p) + \chi_2(p))^2 - 2\chi_1(p) \chi_2(p) = p^{1-\k}\lambda_f(p)^2 - 2.$$
Thus we obtain
\begin{align*}
\langle U_p \phi, U_p \phi \rangle &= \frac{p+1}{p+1} + \frac{p(p-1)p^{1-k}\lambda_f(p)^2}{p+1}\\
    &=  1+ \frac{(p-1)\lambda_f(p)^2}{p+1}p^{2-k},
\end{align*}
 hence we see that by Lemma \ref{classtoadel} this recovers the classical formula from Theorem \ref{classform}.

\section{Applications to $L$-values} \label{Applications to $L$-values}

Let $f \in S_{\kappa}(\Gamma_0(N))$ be a newform.
In this section we apply the results of the previous sections to give a `local' decomposition of the Petersson norm of $f$.  This depends on showing that value $L^{N}(k, \Sym^2 f)$ obtained by meromorphic continuation of $L(s,\Sym^2 f)$ can be expressed as a conditionally convergent Euler product.

Recall that the (partial) \emph{symmetric square $L$-function} of $f$ is defined by the Euler product
\be \label{Ep}
L^N(s, \Sym^2 f) = \prod_{p \nmid N}\frac{1}{L_p(s,\Sym^2 f)},
\ee
where
$$L_{p}(s, \Sym^2 f):=\left(1-\frac{\alpha_{p}^2}{p^{s}}\right)\left(1-\frac{\alpha_{p} \beta_{p}}{p^{s}}\right)\left(1-\frac{\beta_{p}^2}{p^{s}}\right). $$
The product (\ref{Ep}) converges absolutely for $\re s>\kappa$. It is well-known that $L^N(s, \Sym^2 f)$ admits meromorphic continuation to the entire complex plane with possible poles only at $s = \kappa$ and $\kappa - 1$, of order at most one
 \cite[Theorem 1]{ShimuraPLMS75}. In our case (since $f$ is assumed to have trivial character), the $L$-function does not have a pole at $s=\kappa$ \cite[Theorem 2]{ShimuraPLMS75}.  We will continue to denote this extended function by $L^N(s, \Sym^2 f)$. Using that $\alpha_p\beta_p=p^{\kappa-1}$ we conclude that
\begin{equation} \label{form1}
\frac{\langle f_p^{\alpha_p}, f_p^{\alpha_p} \rangle_{Np}}{\langle f, f \rangle_{Np}} = \frac{p^2}{p^2-1}\frac{1}{L_p(\kappa,\Sym^2 f)} =\frac{\zeta_p(2)}{L_p(\kappa, \Sym^2 f)},
\end{equation}
where $\zeta_{p}(s) = 1/(1 - 1/p^s).$

\begin{cor} We have
\begin{equation*}
\langle f_p^{\alpha_p}, f_p^{\alpha_p}\rangle_{Np} = \langle f_p^{\beta_p}, f_p^{\beta_p}\rangle_{Np}.
\end{equation*}
\end{cor}

 Set
\begin{equation*}
\langle f,f\rangle^{(p)}_N:=  \frac{\langle f, f \rangle_{Np}}{\langle f_p^{\alpha_p}, f_p^{\alpha_p} \rangle_{Np}} = \frac{\langle f, f \rangle_{Np}}{\langle f_p^{\beta_p}, f_p^{\beta_p} \rangle_{Np}}.
\end{equation*}
We will now show  that $\langle f,f\rangle^{(p)}_N$ can in some sense be regarded as a `local' (at $p$) period for the symmetric square $L$-function.

\begin{thm}\label{conv1} The value $L^N(\kappa, \Sym^2 f)$ given by the meromorphic continuation is equal to the conditionally convergent Euler product $$\prod_{p \nmid N} \frac{1}{L_{p}(\kappa, \Sym^2 f)}$$ when we order the factors  according to increasing $p$. \end{thm}
\begin{proof} Let $\phi_f$ be defined as in (\ref{adelicdef}) and let $\chi_1(p)=\alpha_p/p^{(\kappa-1)/2}$ and $\chi_2(p)=\beta_p/p^{(\kappa-1)/2}$ be its Satake parameters for $p \nmid N$ as in Section \ref{sec:adeliccalc}. For $p \nmid N$ define $$L_{p}(s, \Sym^2 \phi_f):= \left(1-\frac{\chi_1(p)^2}{p^{s}}\right)\left(1-\frac{1}{p^{s}}\right)\left(1-\frac{\chi_2(p)^2}{p^{s}}\right)$$ and note that $L_{p}(s, \Sym^2 \phi_f)=L_{p}(s+\kappa-1, \Sym^2 f)$. Thus the Euler product $$L^N(s, \Sym^2 \phi_f):=\prod_{p \nmid N}\frac{1}{L_{p}(s, \Sym^2 \phi_f)}$$ converges absolutely for $\re s>1$ and inherits all the corresponding properties (in particular the meromorphic continuation and the lack of a pole at $s=1$) from $L^N(s, \Sym^2 f)$. As before we will continue to denote this extended function by $L^N(s, \Sym^2 \phi_f)$.

Let $\pi$ be the automorphic representation of $\GL_2(\AQ)$ associated with $\phi_f$. It is known  \cite[Theorem 9.3]{GelbartJacquet78} that there exists an automorphic representation $\sigma$ of $\GL_3(\AQ)$ such that the (partial) standard $L$-function $L^N(s, \sigma)$ coincides with $L^N(s, \Sym^2 \pi):=L^N(s, \Sym^2 \phi_f)$. Also $L^N(s, \sigma)$ does not vanish on the line $\re s=1$ by a result of Jacquet and Shalika (see \cite[Theorem 1]{JacquetShalika76}; see also \cite{KohnenSengupta00}).
Finally note that by Lemma \ref{Ramanujan}, we have $|\chi_1(p)| = |\chi_2(p)| = 1$ if $p \nmid N$.
Thus we are in a position to apply Theorem \ref{ethm} in the appendix with $K=\bfQ$ and $d=3$ to $L^N(s, \Sym^2 \phi_f)$  and the theorem follows.
\end{proof}

By Theorem \ref{conv1} and (\ref{form1}) we have $$\frac{L^N(\kappa, \Sym^2 f)}{\prod_{p \nmid N} \langle f,f\rangle^{(p)}_N} = \zeta^N(2),$$
where the superscript means that we omit the Euler factors at primes dividing $N$, and the product $\prod_{p \nmid N}$ (here and below) is ordered according to increasing $p$.
Using \cite[Theorem 5.1]{HidaInvent81} we have $$L^N(\kappa, \Sym^2 f) = \prod_{p\mid N}\left(1-\frac{\lambda_f(p)^2}{p^{\kappa}}\right) \times \frac{2^{2\kappa}\pi^{\kappa+1}}{(\kappa-1)! \delta(N) N \phi(N)}\langle f,f \rangle_N,$$ where $\delta(N)=2$ or 1 according as $N \leq 2$ or not. Using this we obtain the following corollary that can be viewed as a factorization of the `global' period $\langle f,f\rangle_N$ in terms of the `local' periods $ \langle f,f\rangle^{(p)}_N$.

\begin{cor}  We have
\begin{equation*}
\langle f,f \rangle_N = \frac{(\kappa-1)! \delta(N) N \phi(N)\zeta^N(2)}{2^{2\kappa}\pi^{\kappa+1}} \prod_{p\mid N}\frac{1}{1-\lambda_f(p)^2/p^{\kappa}}  \prod_{p \nmid N} \langle f,f\rangle^{(p)}_N.
\end{equation*}
\end{cor}



\appendix
\section{Convergence of Euler products on ${\rm Re}(s) = 1$  \\ by Keith Conrad$^3$}
Let $K$ be a number field.  A degree $d$ Euler product over $K$ is a product
$$
L(s) = \prod_{\mathfrak p} \frac{1}{(1 - \alpha_{\pp,1}\Nm\pp^{-s})\cdots (1 - \alpha_{\pp,d}\Nm\pp^{-s})},
$$
where $|\alpha_{\pp,j}| \leq 1$ for all nonzero prime ideals $\pp$ in the integers of $K$ and $1 \leq j \leq d$.  On the half-plane $\re(s) > 1$ this converges absolutely and is nonvanishing. Combining factors at prime ideals lying over a common prime number,
$L(s)$ is also an  Euler product over $\QQ$ of degree $d[K:\QQ]$.

We want to prove a general theorem about the representability of $L(s)$ by its Euler product on the line $\re(s) = 1$.  If $L(s)$ is the $L$-function of a nontrivial Dirichlet character, this is in \cite[pp.~57--58]{davenport}, \cite[$\S$ 109]{landau}, and \cite[p. 124]{mv} if $s = 1$ and \cite[$\S$ 121]{landau} if $\re(s) = 1$.

\begin{thm}\label{ethm}
If $L(s)$ is a degree $d$ Euler product over $K$
and it admits an analytic continuation to $\re(s) = 1$ where it is nonvanishing, then $L(s)$ is equal to its Euler product on $\re(s) = 1$
when factors are ordered according to prime ideals of increasing norm: if $\re(s) = 1$ then
$$
L(s) =
\lim_{x \rightarrow \infty} \prod_{\Nm\pp \leq x} \frac{1}{(1 - \alpha_{\pp,1}\Nm\pp^{-s})\cdots (1 - \alpha_{\pp,d}\Nm\pp^{-s})}.
$$
\end{thm}

The proof is based on the following lemma about representability of a Dirichlet series on the line $\re(s) = 1$.

\begin{lemma}\label{rthm}
Suppose $g(s) = \sum_{n \geq 1} b_n n^{-s}$ has bounded
Dirichlet coefficients. If
$g(s)$ admits an analytic continuation from $\re(s) > 1$ to $\re(s) \geq 1$, then $g(s)$ is still represented by
its Dirichlet series on the line $\re(s) = 1$.
\end{lemma}

\begin{proof}
See \cite{newman}.
\end{proof}

Here is the proof of Theorem \ref{ethm}.

\begin{proof}
We will apply Lemma \ref{rthm} to a logarithm of $L(s)$, namely the absolutely convergent Dirichlet series
$$
(\log L)(s) := \sum_{\pp} \sum_{k \geq 1} \frac{\alpha_{\pp,1}^k + \cdots + \alpha_{\pp,d}^k}{k\Nm\pp^{ks}},
$$
where $\re(s) > 1$.
The coefficient of $1/\Nm\pp^{ks}$ has absolute value at most $d/k \leq d$, so if we collect terms and write
$(\log L)(s)$ as a Dirichlet series indexed by the positive integers, say $\sum_{n \geq 1} c_n/n^s$, then
$c_n = 0$ if $n$ is not a prime power and $|c_n| \leq d[K:\QQ]$ if $n$ is a prime power.  Therefore
the coefficients of $(\log L)(s)$ as a Dirichlet series over $\ZZ^+$ are bounded.

Since $L(s)$ is assumed to have an analytic continuation to a nonvanishing function on $\re(s) \geq 1$,
$(\log L)(s)$ has an analytic continuation to $\re(s) \geq 1$, so Lemma \ref{rthm} implies that
\begin{equation}\label{L1t}
(\log L)(s) = \sum_{\pp^k}  \frac{\alpha_{\pp,1}^k + \cdots + \alpha_{\pp,d}^k}{k\Nm\pp^{ks}}
\end{equation}
for $\re(s) = 1$, where the terms in the series are collected in order of increasing values of $\Nm(\pp^k)$.

Although a rearrangement of terms in a conditionally convergent series can change its value,
one particular rearrangement of the series in (\ref{L1t}) doesn't change the sum:
\begin{equation}\label{apt}
\sum_{\pp^k} \frac{\alpha_{\pp,1}^k + \cdots + \alpha_{\pp,d}^k}{k\Nm\pp^{ks}} =
\sum_{\pp}\sum_{k \geq 1} \frac{\alpha_{\pp,1}^k + \cdots + \alpha_{\pp,d}^k}{k\Nm\pp^{ks}}
\end{equation}
when $\re(s) = 1$, where the sum on the left is in order of increasing values of $\Nm(\pp^k)$ and
the outer sum on the right is in order of increasing values of $\Nm(\pp)$.
To prove (\ref{apt}), we rewrite it as
\begin{equation}\label{sum2}
 \sum_{\Nm(\pp^k) \leq x}
  \frac{\alpha_{\pp,1}^k + \cdots + \alpha_{\pp,d}^k}{k\Nm\pp^{ks}} =
\sum_{\Nm(\pp) \leq x} \sum_{k \geq 1}
 \frac{\alpha_{\pp,1}^k + \cdots + \alpha_{\pp,d}^k}{k\Nm\pp^{ks}}
  + o(1)
\end{equation}
as $x \rightarrow \infty$, and we will prove (\ref{sum2})
when $\re(s) > 1/2$, not just $\re(s) = 1$.  For $\re(s) = 1$ we can pass to the limit in (\ref{sum2})
as $x \rightarrow \infty$ and conclude (\ref{apt}).

The sum on the right in (\ref{sum2})  the sum on the left in (\ref{sum2}) is equal to
 $$ \sum_{\Nm(\pp) \leq x} \sum_{\substack{k \geq 2\\ \Nm(\pp)^k > x}}
  \frac{\alpha_{\pp,1}^k + \cdots + \alpha_{\pp,d}^k}{k\Nm\pp^{ks}},$$
which is equal to
\begin{equation}\label{bigt}
\sum_{\sqrt{x} < \Nm(\pp) \leq x}
\sum_{k \geq 2}
 \frac{\alpha_{\pp,1}^k + \cdots + \alpha_{\pp,d}^k}{k\Nm\pp^{ks}}
 +
\sum_{\Nm(\pp) \leq \sqrt{x}} \sum_{\substack{k \geq 3 \\ \Nm(\pp)^k > x}}
 \frac{\alpha_{\pp,1}^k + \cdots + \alpha_{\pp,d}^k}{k\Nm\pp^{ks}}.
\end{equation}
The absolute value of the first sum in (\ref{bigt}) is bounded above by
\begin{eqnarray*}
\sum_{\sqrt{x} < \Nm(\pp) \leq x} \sum_{k \geq 2} \left|\frac{\alpha_{\pp,1}^k + \cdots + \alpha_{\pp,d}^k}{k\Nm\pp^{ks}}\right| & \leq &
\sum_{\sqrt{x} < \Nm(\pp) \leq x} \sum_{k \geq 2} \frac{d}{k\Nm\pp^{k\sigma}} \ \ \text{ where } \sigma = \re(s) \\
& < &  \frac{d}{2}\sum_{\sqrt{x} < \Nm(\pp) \leq x} \sum_{k \geq 2} \frac{1}{\Nm\pp^{k\sigma}} \\
& = & \frac{d}{2}\sum_{\sqrt{x} < \Nm(\pp) \leq x} \frac{1}{\Nm\pp^{\sigma}(\Nm\pp^{\sigma} - 1)} \\
& < & \frac{d}{2}\sum_{\sqrt{x} < \Nm(\pp) \leq x} \frac{4}{\Nm\pp^{2\sigma}} \quad \textup{since $\Nm(\mathfrak p)^\sigma > \sqrt{2} > \frac{4}{3}$,}
\end{eqnarray*}
which tends to 0 as $x \rightarrow \infty$ since
$\sum_{\pp} 1/\Nm\pp^{2\sigma}$ converges.
The absolute value of the second sum in (\ref{bigt}) is bounded above by
\begin{align*}
\sum_{\Nm(\pp) \leq \sqrt{x}} \sum_{\substack{k \geq 3\\ \Nm(\pp)^k > x}}  \frac{d}{k\Nm\pp^{k\sigma}} &<
\sum_{\Nm(\pp) \leq \sqrt{x}} \sum_{\substack{k \geq 3\\ \Nm(\pp)^k > x}}  \frac{d}{\log_{\Nm\pp}(x)\Nm\pp^{k\sigma}}\\
    &=
\sum_{\Nm(\pp) \leq \sqrt{x}} \frac{d\log \Nm(\pp)}{\log x}\sum_{\substack{k \geq 3 \\ \Nm(\pp)^k > x}}  \frac{1}{\Nm\pp^{k\sigma}}.
\end{align*}
Letting $n$ be the least integer above $\log_{\Nm\pp}(x)$,
$$
\sum_{\substack{k \geq 3 \\ \Nm(\pp)^k > x}}  \frac{1}{\Nm\pp^{k\sigma}}  = \frac{1/\Nm(\pp)^{n\sigma}}{1 - 1/\Nm(\pp)^{\sigma}} < \frac{1/x^{\sigma}}{1/4} = \frac{4}{x^{\sigma}},
$$
so
$$
\sum_{\Nm(\pp) \leq \sqrt{x}} \sum_{\substack{k \geq 3 \\ \Nm(\pp)^k > x}}  \frac{d}{k\Nm\pp^{k\sigma}} <
\sum_{\Nm(\pp) \leq \sqrt{x}} \frac{4d\log \Nm(\pp)}{x^{\sigma}\log x}  = O\left(\frac{\sqrt{x}}{x^{\sigma}\log x}\right),
$$
which tends to 0 as $x \rightarrow \infty$ since $\sigma > 1/2$.

Now that we established (\ref{apt}), take the exponential of the right side: if $\re(s) = 1$, then
$$
\prod_{\pp} \exp\left(\sum_{k \geq 1} \frac{\alpha_{\pp,1}^k + \cdots + \alpha_{\pp,d}^k}{k\Nm\pp^{ks}}\right) =
\prod_{\pp} \frac{1}{(1 - \alpha_{\pp,1}\Nm\pp^{-s})\cdots (1 - \alpha_{\pp,d}\Nm\pp^{-s})},
$$
where the products run over $\pp$ in order of increasing norms and the
last calculation is justified since $|\alpha_{\pp,j}/\Nm(\pp)^s| \leq 1/\Nm(\pp) < 1$.
Since $L(s) = e^{(\log L)(s)}$ for $\re(s) \geq 1$,
by (\ref{L1t}) and (\ref{apt}) we have
$$
L(s) =
\prod_{\pp} \frac{1}{(1 - \alpha_{\pp,1}\Nm\pp^{-s})\cdots (1 - \alpha_{\pp,d}\Nm\pp^{-s})}
$$
for $\re(s) = 1$, where the product is in order of increasing values of $\Nm\pp$.
\end{proof}

\begin{example}
Let $L(s)$ be the $L$-function of the elliptic curve $y^2 = x^3 - x$ over $\QQ$.  For $\re(s) > 3/2$ it
has an Euler product over the odd primes of the form
\begin{equation}\label{L1p}
L(s) = \prod_{p \not= 2} \frac{1}{1 - a_p p^{-s} + p \cdot p^{-2s}} =
\prod_{p \not= 2} \frac{1}{(1 - \alpha_{p}p^{-s})(1 - \beta_{p}p^{-s})},
\end{equation}
where $|\alpha_{p}| = \sqrt{p}$ and $|\beta_{p}| = \sqrt{p}$ for $p \not= 2$.
Since $y^2 = x^3 - x$ has CM by $\ZZ[i]$,
$L(s)$ is also the $L$-function of a Hecke character $\chi$ on $\QQ(i)$ such that
$|\chi((\alpha))| = |\alpha| =  |\Nm(\alpha)|^{1/2}$ for all nonzero $\alpha$ in $\ZZ[i]$ with odd norm.  Therefore $L(s)$ also
has an Euler product over the nonzero prime ideals of $\ZZ[i]$ of odd norm: for $\re(s) > 3/2$,
\begin{equation}\label{L2p}
L(s) = \prod_{(\pi) \not= (1+i)} \frac{1}{1 - \chi(\pi)/\Nm(\pi)^s}.
\end{equation}

The function $L(s)$ is entire and is nonvanishing on the line $\re(s) = 3/2$,
so $L(s + 1/2)$ fits the conditions of Theorem \ref{ethm} using $K = \QQ$ and $d = 2$ for (\ref{L1p}), and $K = \QQ(i)$ and $d = 1$ for (\ref{L2p}). Therefore (\ref{L1p}) and (\ref{L2p}) are both true
on the line $\re(s) = 3/2$.  For instance,
$L(3/2) \approx .826348$, the partial Euler product for (\ref{L1p}) at $s = 3/2$ over prime numbers up to 100,000 is $\approx .826290$, and
the partial Euler product for (\ref{L2p}) at $s = 3/2$ over nonzero prime ideals in $\ZZ[i]$ with norm
up to 100,000 is $\approx .826480$.
\end{example}


\bibliographystyle{plain}
 \bibliography{mybib}

%
%
%
%
%
%
%

 \end{document}